\documentclass[12pt,reqno]{amsart}
\usepackage{color}
\usepackage{enumerate}
\usepackage{amsthm,amssymb}
\usepackage{amsmath}
\usepackage{graphicx,epsfig}
\usepackage{amsmath}
\usepackage{bbm}
\def\1{\mathbbm{1}}
\usepackage[mathscr]{euscript}
\usepackage{mathrsfs}

\newtheorem{thm}{Theorem}[section]
\newtheorem{lem}[thm]{Lemma}
\newtheorem{prop}[thm]{Proposition}
\newtheorem{cor}[thm]{Corollary}

\theoremstyle{definition}
\newtheorem{ex}[thm]{Example}
\newtheorem{rmk}[thm]{Remark}
\newtheorem{defn}[thm]{Definition}

\def\al{\alpha}

\def\calS{{\mathcal{S}}}

\newcommand{\ba}{\begin{align*}}
\newcommand{\ea}{\end{align*}}
\newcommand{\bal}{\begin{align}}
\newcommand{\eal}{\end{align}}
\newcommand{\be}{\begin{enumerate}}  
\newcommand{\ee}{\end{enumerate}}

\newcommand{\bpr}{\begin{proof}}  
\newcommand{\epr}{\end{proof}}

\begin{document}

\numberwithin{equation}{section}

\title[]{On evolution equations with white-noise boundary conditions}
\author[]{M. Fkirine, S. Hadd, and A. Rhandi}

\address{M. Fkirine and S. Hadd are with Department of Mathematics, Faculty of Sciences, Hay Dakhla, BP8106, Agadir, Morocco;  E-mail (Fkirine) fkirinemohamed@gmail.com, (Hadd) s.hadd@uiz.ac.ma}
\address{A. Rhandi is with Dipartimento di Matematica, Universita degli Studi di Salerno; Via Giovanni Paolo II, 132, Fisciano (Sa), 84084, Italy, E-mail: arhandi@unisa.it}
\thanks{This article is based upon work from COST Action 18232 MAT-DYN-NET, supported by COST (European Cooperation in Science and Technology), www.cost.eu. The third author is member of the Gruppo Nazionale per l'Analisi Matematica, la Probabilità e le loro Applicazioni (GNAMPA) of the Istituto Nazionale di Alta Matematica (INdAM). This work has been supported by the National Center for Scientific and Technical Research (CNRST) via the FINCOME program.}
\vskip 0.5cm
\medskip

\medskip

\bigskip
\begin{abstract}
In this paper, we delve into the study of evolution equations that exhibit white-noise boundary conditions. Our primary focus is to establish a necessary and sufficient condition for the existence of solutions, by utilizing the concept of admissible observation operators and the Yosida extension for such operators. By employing this criterion, we can derive an existence result, which directly involves the Dirichlet operator. In addition, we also introduce a Desch-Schappacher perturbation result, which proves to be instrumental in further understanding these equations. Overall, our paper presents a comprehensive analysis of evolution equations with white-noise boundary conditions, providing new insights and contributing to the existing body of knowledge in this field.
\end{abstract}
\maketitle
\section{Introduction}
In this paper, our primary focus is on equations that exhibit boundary white-noise conditions of the form:
	\begin{align}\label{00}
		\begin{cases}
			\dot{X}(t)=A_mX(t),& t\in[0,T],\\
			GX(t)=\dot{W}(t), &t\in[0,T],\\
			X(0)=X_0,
		\end{cases}
	\end{align}
where $A_m$ is a linear closed operator that maps from $Z$, a Hilbert space that is continuously and densely embedded in the separable Hilbert space $H$. We consider $(W(t))_{t\in[0,T]}$ as a cylindrical Wiener process over $U$, another separable Hilbert space. The operator $G:Z\to U$ is assumed to be linear.

Our paper aims to provide a comprehensive analysis of these equations, with a particular emphasis on establishing the existence of solutions. To achieve this, we utilize various techniques such as admissible observation operators and the Yosida extension. Our research also presents a Desch-Schappacher perturbation result that provides a deeper understanding of these equations. Overall, our work contributes to the existing literature on boundary white-noise conditions and sheds light on the underlying mathematical properties of these equations.

The equations of the form \eqref{00} were first studied by Balakrishnan \cite{bala} and further developed by Da Prato and Zabczyk \cite{da1993stochastic}, with extensive coverage in \cite[Chap. 13]{da1996ergodic}. Since then, several authors have shown interest in this problem, including \cite{gozzi}, \cite{debussche}, \cite{fabbri}, \cite{masiero}, \cite{guatteri}, \cite{brzezniak}, and \cite{goldys}.

In the literature (e.g., \cite{da1993stochastic}), the standard conditions for \eqref{00} are that $A:=A_m$ with $D(A) = \ker(G)$ is the generator of a strongly continuous semigroup $\mathbb{T}$ on $H$, and that the stationary boundary value problem
\begin{align*}
(A_m-\lambda)z=0, \quad Gz=u,
\end{align*}
has a unique solution $z=\mathbb{D}_\lambda u\in H$ for some $\lambda$ and all $u\in U$. Here, $\mathbb{D}_\lambda$ is the Dirichlet operator associated with $A_m$ and $G$. Using this operator, the problem \eqref{00} can be equivalently expressed as the stochastic Cauchy problem	
\begin{align*}
		\bf(SCP)_{A,B}\;\;
		\begin{cases}
			dX(t)=A_{-1}X(t)dt+BdW(t), \quad t\in[0,T],\\
			X(0)=X_0,
		\end{cases}
	\end{align*}
in $H_{-1}$, where $H_{-1}$ is the extrapolation space associated with $A_m$ and $H$, $A_{-1}$ with domain $D(A)=H$ is the extension of $A$ to $H$, and $B:=(\lambda-A_{-1})\mathbb{D}_{\lambda}\in\mathcal{L}(U,H_{-1})$, as explained in Section \ref{section2}. We define a solution of $\bf(SCP)_{A,B}$ as any $H$-valued process $X(\cdot)$ satisfying
\begin{align*}
X(t)=\mathbb{T}(t)X_0+\int_0^t\mathbb{T}_{-1}(t-s)BdW(s), \qquad t\in[0,T],
\end{align*}
where $\mathbb{T}_{-1}$ is the extrapolation semigroup associated with $\mathbb{T}$ and $H$, which is the semigroup generated by $A_{-1}$. This paper aims to extend the existing literature on this topic by providing new insights into these equations and establishing necessary and sufficient conditions for the existence of solutions.

To study the unique mild solution of $\bf(SCP)_{A,B}$ and \eqref{00}, we introduce the input maps $\Phi_t:L^2([0,t];U)\to H_{-1}$, associated with the control operator $B$ and $A$ (refer to Section \ref{section2} for details). The existence of a unique mild solution for $\bf(SCP)_{A,B}$ and \eqref{00} is guaranteed if and only if $\Phi_t$ admits an extension to a Hilbert-Schmidt operator from $L^2([0,t];U)$ to $H$. This result was first established in Hilbert spaces by Da Prato and Zabczyk in \cite{da1993stochastic}, and later refined in \cite{Vanneerven2013} for Banach spaces. Note that the notion of Hilbert-Schmidt operators is an essential tool in the study of infinite-dimensional systems. In the context of $\bf(SCP)_{A,B}$ and \eqref{00}, the requirement of an extension of $\Phi_t$ to a Hilbert-Schmidt operator is a strong condition, ensuring the existence and uniqueness of solutions.

Let us consider the perturbed stochastic Cauchy problem obtained by adding a linear bounded operator $\mathscr{P}:H\to H_{-1}$ to the generator $A$ in the original problem: 
\begin{align*}
		\bf(SCP)_{A+\mathscr{P},B}\;\;
		\begin{cases}
			dX(t)=\left(A_{-1}+\mathscr{P}\right)X(t)dt+BdW(t), \quad t\in[0,T],\\
			X(0)=X_0,
		\end{cases}
	\end{align*}
The first question that arises is determining the conditions under which the operator $(A_{-1}+\mathscr{P})_{|H}$ is a generator in $H$. This question has already been addressed by Desch and Schappacher, who introduced an appropriate class of perturbations in their work \cite{desch} and in \cite[Chapter III.3.a]{nagel}. Specifically, if $\mathscr{P}\in\mathcal{L}(H,H_{-1})$ is an admissible control operator for $A$ (as defined in Section \ref{section2}), then the part of $A_{-1}+\mathscr{P}$ in $H$ generates a strongly continuous semigroup $\mathscr{T}:=(\mathscr{T}(t))_{t\geq0}$ on $H$.

Despite the importance of studying generator perturbation of stochastic Cauchy problems, the existing literature on this topic is limited. Only a few results have been reported in the literature, namely \cite{Peszat1992}, \cite{haak} and \cite{lahbiri2021}. In \cite{Peszat1992}, Peszat studied the existence and uniqueness of solutions, as well as the absolute continuity of the laws of the solutions of the problem $\bf(SCP)_{A+P,I}$ where the operator $P\in\mathcal{L}(D(A),H)$ is closed, $$\bigcup_{t>0}{\rm Ran} \mathbb{T}(t)\subset D(A)\quad\text{ and}\quad \int_0^T\|P\mathbb{T}(t)\|^2dt<\infty.$$ In \cite{lahbiri2021}, a variation of constant formula for $\bf(SCP)_{A+P,B}$ is given when $\left(W(t)\right)_{t\geq0}$ is a one-dimensional Wiener process and $B\in\mathcal{L}(U,H)$, see also \cite{fkirine2021}. In \cite{haak}, the authors studied the existence of solutions, as well as the invariant measure of $\bf(SCP)_{A+P,B}$ in the context of Banach spaces for $P\in\mathcal{L}(H)$ and $B\in\mathcal{L}(U,H)$.

This paper has two main goals. The first one is to provide a characterization of the existence and uniqueness of solutions to the stochastic Cauchy problem $\bf(SCP)_{A,B}$ using the notion of admissible observation operators and the Yosida extension of unbounded linear operators (see Section \ref{section2}). This characterization will allow us to derive a necessary and sufficient condition for the existence and uniqueness of solutions in terms of the Dirichlet operator, which is more accessible than the condition on the input maps, since it only depends on the operator $B$ and the resolvent of $A$.

The second aim of this paper is to establish the existence and uniqueness of solutions to $\bf(SCP)_{A+\mathscr{P},B}$. To do so, we assume that $\bf(SCP)_{A,B}$ has a solution and $\mathscr{P}$ is an admissible control operator for $A$. However, both the operators $B$ and $\mathscr{P}$ are unbounded, so we need to address this issue. Specifically, we prove that $\bf(SCP)_{A+\mathscr{P},B}$ has a solution if and only if the input maps associated with $(A_{-1}+\mathscr{P})_{|H}$ and $B$ are Hilbert-Schmidt from $L^2([0,T];U)$ to $H$. Unfortunately, we do not know an explicit expression for the extrapolation semigroup of $\mathscr{T}$ and $H$ in terms of the semigroup $\mathbb{T}$. Therefore, we introduce a new notion of $\mathcal{S}$-admissible observation operators (see Definition \ref{def3.3}) to address this problem using a dual approach.

The rest of the paper is organized as follows: In Section \ref{section2}, we present some preliminaries about admissible control and observation operators, as well as the results needed in the following sections.

Section \ref{section3} presents sufficient and necessary conditions for the existence of solutions to $\bf(SCP)_{A,B}$. In Proposition \ref{prop3}, we use the notion of admissible control/observation operators and the Yosida extensions of such operators. In Proposition \ref{car_adm}, we provide a condition that characterizes the existence and uniqueness of solutions directly in terms of the Dirichlet operator.

In Section \ref{section4}, we address the second aim of the paper and prove the existence and uniqueness of solutions to the perturbed stochastic Cauchy problem $\bf(SCP)_{A+\mathscr{P},B}$.

The last section is devoted to an application to a one-dimensional heat equation with a white-noise Neumann boundary control.

{\bf Notation.} Throughout the paper, we use the following notation. $H$, $U$, and $Y$ always denote separable Hilbert spaces. For $T>0$, we denote the spaces $L^2(0,T;U)$ and $L^2(0,T;Y)$ by $\mathcal{U}_T$ and $\mathcal{Y}_T$, respectively. The scalar product is denoted by $\langle \cdot, \cdot \rangle$, equipped with a subscript to specify the space if necessary. $\mathcal{L}(U,H)$ denotes the space of bounded linear operators from $U$ to $H$, endowed with the usual operator norm. $\mathcal{L}_2(U,H)$ denotes the space of Hilbert-Schmidt operators from $U$ to $H$. We use $\|\cdot\|_{\mathcal{L}_2(U,H)}$ to denote the Hilbert-Schmidt norm for these operators or simply $\|\cdot\|_2$ when there is no ambiguity about the spaces. We denote by $\mathbb{C}_\alpha$ the half-plane of all $\lambda\in\mathbb{C}$ with $\text{Re}(\lambda)>\alpha$. Given a separable Hilbert space $\mathscr{X}$, $H^2(\mathscr{X})$ denotes the Hardy space of all analytic functions $f:\mathbb{C}_0\to \mathscr{X}$ that satisfy
\begin{align*}
		\sup_{\alpha>0}\int_{-\infty}^{+\infty}\|f(\alpha+i\beta)\|^2d\beta<\infty.
	\end{align*}
	The norm of $f$ in this space is, by definition,
	\begin{align*}
		\|f\|_{H^2}=\frac{1}{2\pi}\left(\sup_{\alpha>0}\int_{-\infty}^{+\infty}\|f(\alpha+i\beta)\|^2d\beta\right)^{\frac{1}{2}}.
	\end{align*}
In this work, we require some classical background on semigroup theory. Let $A$ be the generator of a strongly continuous semigroup $\mathbb{T} = (\mathbb{T}(t))_{t\geq0}$ on a Hilbert space $H$, with growth bound $\omega_0(\mathbb{T})$. We denote by $H_1$ the space $D(A)$ with the norm $\|x\|_1 = \|(\beta - A)x\|$, where $\beta$ is an arbitrary (but fixed) element in the resolvent set $\rho(A)$. We also denote by $H_{-1}$ the completion of $H$ with respect to the norm $\|x\|_{-1} = \|R(\beta, A)x\|$, where $R(\beta, A) := (\beta - A)^{-1}$, and $\beta$ is as before. It is easy to verify that $H_1$ and $H_{-1}$ are independent of the choice of $\beta$, since different values of $\beta$ lead to equivalent norms on $H_1$ and $H_{-1}$. Note that the norm $\|\cdot\|_1$ is equivalent to the graph norm of $A$. We have $H_1 \subset H \subset H_{-1}$ densely and with continuous embedding. The semigroup $\mathbb{T}$ extends to a strongly continuous semigroup on $H_{-1}$, whose generator is an extension of $A$, with domain $H$. We denote the extensions of $\mathbb{T}$ and $A$ by $\mathbb{T}_{-1} := (\mathbb T_{-1}(t))_{t\geq0}$ and $A_{-1}$, respectively.

Furthermore, we use $H_1^d$ to denote $D(A^\ast)$ with the norm $\|x\|_1^d = \|(\bar{\beta} - A^\ast)x\|$ and $H_{-1}^d$ to denote the completion of $H_1^d$ with respect to the norm $\|x\|_{-1}^d = \|R(\bar{\beta},A^\ast)x\|$, where $A^\ast$ is the adjoint operator of $A$ and $\bar{\beta}$ is an arbitrary (but fixed) element in the resolvent set $\rho(A^\ast)$. Note that $H_{-1}$ is the dual of $H_1^d$ with respect to the pivot space $H$.
\section{Background on admissible control/observation operators}\label{section2}
In this section, we present general information about admissible control and observation operators. The content of this section can be found in greater detail, along with numerous references, in \cite{weiss1989admissible}, \cite{weiss1989admissiblea}, and \cite[Chapter 4]{tucsnak2009observation}.

Throughout this section, we assume that $H$, $U$, $Y$, and $Z$ are separable Hilbert spaces, with $Z$ continuously and densely embedded in $H$, and $T>0$ is a real number. We consider the following boundary control problem:
\begin{align}\label{conpro}
		\begin{cases}
			\dot{x}(t)=A_mx(t), & t\in[0,T], \\
			Gx(t)=u(t), &t\in[0,T],\\
			x(0)=x_0\in H,
		\end{cases}
	\end{align}
	where $u\in \mathcal{U}_T$ is a control function, $A_m:Z\subset H\to H $ is a closed linear operator and $G:Z\to U$ is a linear operator such that the following assumptions
	\begin{enumerate}
		\item[\bf(A1)] The operator $G:Z\to U$ is surjective,
		\item[\bf(A2)] The operator $A:=(A_m)_{|D(A)}$ with $D(A):=\ker G$ generates a $C_0$-semigroup $\mathbb{T}:=(\mathbb{T}(t))_{t\geq0}$ on $H$,
	\end{enumerate}
	are satisfied. It is shown by Greiner \cite[Lemmas 1.2, 1.3]{Greiner} that under the assumptions {\bf(A1)} and {\bf(A2)}, the domain $D(A_m)$ can be viewed as the direct sum of $ D(A)$ and $\ker(\lambda-A_m)$ for any $\lambda\in\rho(A)$. Moreover, the operator $G_{|\ker(\lambda-A_m)}$ is invertible and the inverse
	\begin{align*}
		\mathbb{D}_\lambda:=\left(G_{|\ker(\lambda-A_m)}\right)^{-1}: U\to \ker(\lambda-A_m)\subset H, \qquad \lambda\in\rho(A),
	\end{align*}
	is bounded. The operator $\mathbb{D}_\lambda$ is called the Dirichlet operator associated with $A_m$ and $G$. We consider the following operator
	\begin{align*}
		B:=(\lambda-A_{-1})\mathbb{D}_\lambda\in\mathcal{L}(U,H_{-1}), \qquad \lambda\in\rho(A).
	\end{align*}
	As for any $u\in U$ and  $\lambda\in\rho(A)$ we have $\mathbb{D}_\lambda u\in \ker (\lambda-A_m)$ so $\lambda\mathbb{D_\lambda}u=A_m\mathbb{D}_\lambda u$. Hence,
	\begin{align*}
		\left(A_m-A_{-1}\right)\mathbb{D}_\lambda u=(\lambda-A_{-1})\mathbb{D}_\lambda u=Bu.
	\end{align*}
	Since $\mathbb{D}_\lambda$ is the inverse of $G_{|\ker(\lambda-A_m)}$ and $D(A_m)=D(A)\oplus \ker(\lambda-A_m)$, it follows that
	\begin{align*}
		A_m=(A_{-1}+BG)_{|Z}.
	\end{align*}
	Thus, the boundary control system \eqref{conpro} can be reformulated as
	\begin{align}\label{conpro1}
		\begin{cases}
			\dot{x}(t)=A_{-1}x(t)+Bu(t), &t\in[0,T],\\
			x(0)=x_0\in H.
		\end{cases}
	\end{align}
	The mild solution of the equation \eqref{conpro1} is given by
	\begin{align*}
		x(t)=\mathbb{T}(t)x_0+\Phi_tu, \qquad u\in \mathcal{U}_T,\;\;x_0\in H,
	\end{align*}
	where $\Phi_t\in\mathcal{L}(\mathcal{U}_T,H_{-1})$ is defined by
	\begin{align}\label{inp-map}
		\Phi_tu:=\int_0^t \mathbb{T}_{-1}(t-s)Bu(s)ds.
	\end{align}
	Notice that in the above formula, $\mathbb{T}_{-1}$ acts on $H_{-1}$ and the integration is carried out in $H_{-1}$. This motivates the following definition.
\begin{defn}
		The operator $B\in\mathcal{L}(U,H_{-1})$ is called an admissible control operator
		for $A$ if \text{Ran}$(\Phi_\tau)\subset H$ for some $\tau>0$.
	\end{defn}
	It is worth noting that if $B$ is an admissible control operator for $A$, then the closed graph theorem guarantees that $\Phi_t\in\mathcal{L}(\mathcal{U}_T,H)$ for all $t\ge 0$. As a result, for any $u\in \mathcal{U}_T$ and $x_0\in H$, the solutions $x(\cdot)$ of \eqref{conpro1} stay in $H$ and form a continuous $H$-valued function of $t$. The operators $\Phi_t$ are commonly referred to as input maps associated with the pair $(A,B)$.
	
	Now we deal with the concept of admissible observation operators, a dual concept of admissible control operators when we work in reflexive Banach spaces, in particular Hilbert spaces. This concept is introduced and developed in \cite{weiss1989admissible}. We have the following definition:
\begin{defn}
		The operator $C\in\mathcal{L}(H_1,Y)$ is called an admissible observation operator for $A$ if for some (hence all) $\al>0$ there is a constant $\gamma:=\gamma(\al)>0$ such that
		\begin{align*}
			\int_0^\al \|C\mathbb{T}(t)x\|^2dt\leq \gamma^2\|x\|^2
		\end{align*}
for any $x\in D(A)$.
\end{defn}
Suppose $C\in\mathcal{L}(H_1,Y)$ is an admissible observation operator for $A$. Then,  the map $\Psi_T$ defined by
\begin{equation*}
(\Psi_T x)(t)=C \mathbb T(t)x, \quad x\in D(A), \quad t\in [0,T],
\end{equation*}
has an extension to an operator $\Psi_T\in\mathcal{L}(H,\mathcal{Y}_T)$. The operators $\Psi_T$ are called output maps corresponding to the pair $(A,C)$.

As shown in \cite[Section 4.4]{tucsnak2009observation},  $B\in\mathcal{L}(U,H_{-1})$ is an admissible control operator for $\mathbb{T}$ if and only if $B^*\in\mathcal{L}\left(H_1^d,U\right)$ is an admissible observation operator for the dual semigroup $\mathbb{T}^*$. Moreover,  for every $T>0$ the adjoint $\Phi_T^*\in\mathcal{L}(H,\mathcal{U}_T)$ of the operator $\Phi_T$ introduced in \eqref{inp-map} is given by
	\begin{align*}
		\left(\Phi_T^*x\right)(t)=\left(\Psi_T^dx\right)(T-t), \qquad t\in[0,T],\;\;x\in H,
	\end{align*}
	where $(\Psi_T^d)_{T\geq0}$ are the output maps corresponding to the pair $(A^*, B^*)$.
	
	Now, we introduce the Yosida extension of $C$, denoted $C_\Lambda$, by
	\begin{align*}
		&C_\Lambda x=\lim_{\lambda\to+\infty}C\lambda R(\lambda,A)x,\\ &D(C_\Lambda)=\left\{x\in H, \quad \lim_{\lambda\to+\infty}C\lambda R(\lambda,A)x\;\;\text{exists in}\; Y\right\}.
	\end{align*}
	Clearly, $H_1\subset D(C_\Lambda)\subset H$. We note that if $C$ is an admissible observation operator for $A$, the representation theorem of Weiss \cite[Theorem 4.5]{weiss1989admissible}, shows that \text{Ran}$(\mathbb{T}(t))\subset D(C_\Lambda)$ and
	\begin{align*}
		(\Psi_T x)(t)=C_\Lambda \mathbb{T}(t)x,
	\end{align*}
	for all $x\in H$ and for almost every $t\in (0,T]$.

\section{Equations with white-noise boundary conditions}\label{section3}
	Let $(\Omega,\mathcal{F},\mathbb{P})$ be a probability space with right continuous increasing family $\mathbf{F}=(\mathcal{F}_t)_{t\geq0}$ of sub-$\sigma$-fields of $\mathcal{F}$ each containing $\mathbb{P}$-null sets. Let $(e_n)_{n\in\mathbb{N}}$ be an orthonormal basis in $U$ and let $\left\{\beta_n\right\}$ be a sequence of independent real valued $\mathbf{F}$-Wiener processes. We define a cylindrical Wiener process on $U$ by the serie
	\begin{align*}
		W(t)=\sum_{n=0}^\infty\beta_n(t)e_n, \quad t\geq0,
	\end{align*}
	which converges in a Hilbert space $\tilde{U}$ containing $U$ with a Hilbert-Schmidt embedding.
	
	We consider the following problem with boundary white-noise condition
	\begin{align}\label{BWNC}
		\begin{cases}
			dX(t)=A_mX(t), &t\in[0,T],\\
			GX(t)=\dot{W}(t), &t\in[0,T],\\
			X(0)=X_0\in H.
		\end{cases}
	\end{align}
	
	Under the assumptions {\bf(A1)} and {\bf(A2)}, Da Prato and Zabczyk \cite{da1993stochastic} (see also \cite[Chap. 13]{da1996ergodic}) proved that the boundary problem \eqref{BWNC} is reformulated as the following stochastic Cauchy problem
	\begin{align*}
		{\bf(SCP)_{A,B}}
		\begin{cases}
			dX(t)=A_{-1}X(t)+BdW(t), &t\in[0,T],\\
			X(0)=X_0\in H,
		\end{cases}
	\end{align*}
	in $H_{-1}$, where $B\in\mathcal{L}(U,H_{-1})$ is the control operator associated with $A_m$ and $G$.
	
	It is well known, see \cite[Theorem 5.4]{da2014stochastic}, that ${\bf(SCP)_{A,B}}$ has a unique mild solution
	\begin{align*}
		X(t)=\mathbb{T}(t)X_0+\int_0^t\mathbb{T}_{-1}(t-s)BdW(s), \qquad t\in[0,T],
	\end{align*}
	in $H_{-1}$ if and only if  for some $T>0$
	\begin{align*}
		\int_0^T\left\|\mathbb{T}_{-1}(t)B\right\|_{\mathcal{L}_2(U,H_{-1})}^2dt<\infty.
	\end{align*}
	
	We mention that in the above formula $\mathbb{T}_{-1}$ acts on $H_{-1}$ and then the stochastic convolution  carried out in $H_{-1}$. Here, we are interested in $H$-valued solutions to the stochastic Cauchy problem ${\bf(SCP)_{A,B}}$. This motivates the following definition.
	\begin{defn}
		Assume that ${\bf(SCP)_{A,B}}$ has a mild solution $X$ in $H_{-1}$. We say that ${\bf(SCP)_{A,B}}$ has a mild solution if for each $t\in[0,T]$ the process $X(t)$ takes values in $H$.
	\end{defn}
	The weak solution of ${\bf(SCP)_{A,B}}$ can also be defined using the same terminology. In fact, $X$ is a mild solution of ${\bf(SCP)_{A,B}}$ if and only if it is a weak solution of ${\bf(SCP)_{A,B}}$. It is worth mentioning that if a solution exists, it is unique. For more details about this fact, see \cite[Proposition 2.4]{Vanneerven2013}. Therefore, to simplify the terminology, we will simply speak of a \textit{solution}.
	
	The following theorem from \cite{da1993stochastic} (in Hilbert spaces) and \cite[Proposition 2.4]{Vanneerven2013} (in Banach spaces) gives a necessary and sufficient condition for the existence and uniqueness of a solution for ${\bf(SCP)_{A,B}}$.
	\begin{thm}\label{TH32}
		The stochastic Cauchy problem ${\bf(SCP)_{A,B}}$ has a solution if and only if $\Phi_T$ is Hilbert-Schmidt from $\mathcal{U}_T$ to $H$.
	\end{thm}
	To characterize the existence of  solutions to ${\bf(SCP)_{A,B}}$ using the notion of admissibility introduced in Section \ref{section2}, we introduce the following definition.
	\begin{defn}\label{def3.3}
		Let $B\in\mathcal{L}(U,H_{-1})$ and $C\in\mathcal{L}(H_1,Y)$.
		\begin{enumerate}
			\item[(i)] The operator $B$ is said to be $\mathcal{S}$-admissible control operator for $A$ if the operator $\Phi_T$ has an extension to a Hilbert-Schmidt operator from $\mathcal{U}_T$ into $H$.
			\item[(ii)] The operator $C$ is said to be $\mathcal{S}$-admissible observation operator for $A$ if the operator $\Psi_T$ has an extension to a Hilbert-Schmidt operator from $H$ into $\mathcal{Y}_T$.
		\end{enumerate}
	\end{defn}
	Since every Hilbert-Schmidt operator is bounded, it follows that the $\mathcal{S}$-admissibility implies the admissibility. For the converse implication we have the following proposition.
	\begin{prop}\label{prop3}
		Suppose that $B\in\mathcal{L}(U,H_{-1})$ and $C\in\mathcal{L}(H_1,Y)$.
		\begin{enumerate}
			\item[(i)] $C$ is $\mathcal{S}$-admissible observation operator for $A$ if and only if $C$ is an admissible observation operator for $A$ and
			\begin{equation*}
				\gamma(T):=\int_0^T\|C_\Lambda \mathbb{T}(t)\|^2_2dt<\infty
			\end{equation*}
			for some $T>0$.
			\item[(ii)] $B$ is $\mathcal{S}$-admissible control operator for $A$ if and only if $B^*$ is $\mathcal{S}$-admissible observation operator for $A^*$.
			\item[(iii)] $B$ is $\mathcal{S}$-admissible control operator for $A$ if and only if $B^*$ is an admissible observation operator for $A^*$ and
			\begin{equation}\label{cond}
				\int_0^T\|B^*_\Lambda \mathbb{T}^*(t)\|^2_2dt<\infty
			\end{equation}
			for some $T>0$.
		\end{enumerate}
	\end{prop}
	\begin{proof}
		If $C\in\mathcal{L}(H_1,Y)$ is $\mathcal{S}$-admissible, then $\Psi_T$ is linear bounded from $H$ to $\mathcal{Y}_T$. Therefore, $C$ is an admissible observation for $A$ and $(\Psi_T x)(t)=C_\Lambda \mathbb{T}(t)x$ for a.e $t\geq0$ and $x\in H$. Now, let $(e_k)_{k\in\mathbb{N}}$ be an orthonormal basis of $H$. Then, we have
		\begin{eqnarray*}
			\|\Psi_T\|^2_2&=& \sum_{k\in\mathbb{N}}\|\Psi_T e_k\|^2\\
			&=& \sum_{k\in\mathbb{N}}\int_0^T\|\left(\Psi_T e_k\right)(t)\|^2dt\\
			&=&\sum_{k\in\mathbb{N}}\int_0^T\|C_\Lambda \mathbb{T}(t)e_k\|^2dt\\
			&=& \int_0^T \|C_\Lambda\mathbb{T}(t)\|^2_2dt.
		\end{eqnarray*}
		Thus (i) is satisfied. Moreover, it is shown in \cite[Theorem 4.4.3]{tucsnak2009observation} that $B$ is an admissible control operator for $A$ if and only if $B^*$ is an admissible observation operator for $A^*$. In this case, we have
		\begin{equation*}
			\left(\Phi_T^*x\right)(t)=B_\Lambda^*\mathbb T^*\left(T-t\right)x,
		\end{equation*}
		for every $x\in H$ and a.e $t\in[0,T]$. Thus,
		\begin{equation*}
			\int_0^T\|B^*_\Lambda \mathbb{T}^*(t)\|^2_2dt=\|\Phi_T^*\|_2^2.
		\end{equation*}
		Consequently, (ii) and (iii) follow from (i).
	\end{proof}
	It follows from (iii) of the above proposition that the stochastic Cauchy problem  ${\bf(SCP)_{A,B}}$  has solution if and only if \eqref{conpro1}  has a solution and the condition \eqref{cond} holds.
	\begin{rmk}\label{rem1}
		Let $C$ be an $\mathcal{S}$-admissible observation operator for $A$ and let $(\mathbb{T}(t))_{t\geq0}$ be exponentially stable, then
		$$\int_0^\infty\|C_\Lambda \mathbb{T}(t)\|^2_2dt<\infty.$$
		In fact, we choose $t_0>0$ large enough such that $\|\mathbb T(t_0)\|<1$. Then,
		\begin{eqnarray*}
			\int_0^\infty\|C_\Lambda \mathbb{T}(t)\|^2_2dt&=&\sum_{k=0}^{\infty}\int_{kt_0}^{(k+1)t_0}\|C_\Lambda \mathbb{T}(t)\|_2^2dt\\
			&=&\sum_{k=0}^{\infty} \int_0^{t_0}\|C_\Lambda \mathbb{T}(t)\mathbb T(kt_0)\|_2^2dt\\
			&\leq&\gamma(t_0)\sum_{k=0}^{\infty} \|\mathbb T(t_0)\|^{2k}<\infty.
		\end{eqnarray*}
	\end{rmk}
	
	\begin{prop}\label{Weiss-con}
		Assume that $C\in\mathcal{L}(H_1,Y)$ is $\mathcal{S}$-admissible observation operator for $A$. Then for every $\omega>\omega_0(\mathbb{T})$ there exists $K_\omega\geq 0$ such that
		\begin{equation}\label{con}
			\|CR(\lambda,A)\|_2\leq \frac{K_\omega}{\sqrt{Re(\lambda)-\omega}}, \qquad \forall \lambda\in \mathbb{C}_\omega.
		\end{equation}
	\end{prop}
	\begin{proof}
		Let $\lambda\in \mathbb{C}$ with $Re(\lambda)>\omega_0(\mathbb{T})$. We choose $\omega\in(\omega_0(\mathbb{T}),Re(\lambda))$ and set $\epsilon= Re(\lambda)-\omega$. Let $(e_k)_{k\in \mathbb{N}}$ be an orthonormal basis of $H$. It follows from \cite[Theorem 4.3.7]{tucsnak2009observation} that the Laplace transform of $\Psi_Tz$ exists at $\lambda$ for any $z\in H$ and it is given by $\widehat{(\Psi_T z)}(\lambda)=CR(\lambda,A)z$. Then, we have
		\begin{eqnarray*}
			\|CR(\lambda,A)\|_2^2&=& \sum_{k=1}^\infty\|CR(\lambda,A)e_k\|^2\\
			&=& \sum_{k=1}^\infty \left\|\int_0^\infty e^{-\lambda t}C_\Lambda \mathbb{T}(t)e_kdt\right\|^2\\
			&\leq& \sum_{k=1}^\infty \left(\int_0^\infty |e^{-\epsilon t}|\cdot\|e^{-\omega t}C_\Lambda\mathbb{T}(t)e_k\|dt\right)^2.
		\end{eqnarray*}
		It follows from \cite[Proposition 4.3.6]{tucsnak2009observation}, that  $t\mapsto e^{-\omega t}C_\Lambda \mathbb{T}(t)e_k\in L^2(0,+\infty;Y)$. Then, the Cauchy-Schwarz inequality implies
		\begin{eqnarray*}
			\|CR(\lambda,A)\|_2^2&\leq \int_0^\infty |e^{-\epsilon t}|^2dt\cdot \sum_{k=1}^\infty \int_0^\infty \|e^{-\omega t}C_\Lambda \mathbb{T}(t)e_k\|^2dt.
		\end{eqnarray*}
		Consequently, Remark \ref{rem1} implies that there exists $K_\omega\geq0$ such that
		\begin{equation*}
			\|CR(\lambda,A)\|_2\leq \frac{K_\omega}{\sqrt\epsilon}= \frac{K_\omega}{\sqrt{Re(\lambda)-\omega}}.
		\end{equation*}
	\end{proof}
	\begin{rmk}
		The converse implication of the above proposition is known as the stochastic Weiss conjecture:
		\begin{enumerate}
			\item[(i)] 	It is shown in \cite[Theorem 4.5]{Jacob2003}, that if $\mathbb{T}$ is a contraction semigroup on $H$, $Y$ is a Hilbert space and $C\in\mathcal{L}(H_1,Y)$ satisfies the estimate \eqref{con} for every $\lambda\in \mathbb{C}_\omega$ then $C$ is an admissible observation operator for $A$. We ask if $C$ can be an $\mathcal{S}$-admissible observation operator for $A$.
			\item[(ii)] A dual version of the estimate \eqref{con} was considered in \cite{Vanneerven2013}.  In fact, assuming that $-A$ is sectorial, injective with dense range and admits a bounded $H^\infty$-calculus of angle less than $\frac{\pi}{2}$ on $H$, the authors showed that if for any $\lambda>0$,
			\begin{align*}
				\|R(\lambda,A_{-1})B\|_2<\infty\quad \text{and}\quad \sum_{n\in\mathbb{Z}}2^n\|R(2^n,A_{-1})B\|^2_2<\infty,
			\end{align*}
			then $B$ is an  $\mathcal{S}$-``infinite" admissible control operator for $A$, that is, the operator $\widetilde{\Phi}$ defined by 
			$$\widetilde{\Phi}u:=\lim_{T\to \infty}\int_0^T \mathbb{T}_{-1}(t)Bu(t)\,dt,\quad u\in L^2(0,\infty ;U),$$
			exists in $H$ and it is  Hilbert-Schmidt from $L^2(0,+\infty;U)$ to $H$. 
		
		\end{enumerate}
	\end{rmk}
	Now, we give a criterion of the $\mathcal{S}$-admissibility of $B$ in terms of the Dirichlet operator.
	\begin{prop}\label{car_adm}
		Let $\omega >\omega_0(\mathbb{T})$.  Then we have the following:
		\begin{enumerate}
			\item[(i)] $C$ is $\mathcal{S}$-admissible for $A$ if and only if
			\begin{equation}\label{estim}
				\sum_{n\in\mathbb{Z}}\left\|CR\left(\omega+\frac{2\pi in}{T},A\right)\right\|^2_2<\infty.
			\end{equation}
			\item[(ii)] $B$ is $\mathcal{S}$-admissible for $A$ if and only if
			\begin{equation*}
				\sum_{n\in\mathbb{Z}}\left\|\mathbb{D}_{\omega+\frac{2\pi in}{T}}\right\|^2_2<\infty.
			\end{equation*}
		\end{enumerate}
	\end{prop}
	\begin{proof}
		(i) Let $\omega>\omega_0(\mathbb{T})$. Define the operator $A_\omega:=A-\omega$ with domain $D(A_\omega)=D(A)$ and $\Psi_T^\omega$  be the output map  associated to  $(A_\omega,C)$. Now, let $(e_k)_{k\in\mathbb{N}}$ and $(h_j)_{j\in\mathbb{N}}$ be an orthonormal basis of $H$ and $Y$, respectively. Choose $(f_n:=e^{-\frac{2\pi n i}{T}(\cdot)})_{n\in\mathbb{Z}}$ as an orthonormal basis of $L^2(0,T)$. Then, $(e^{-\frac{2\pi n i}{T}(\cdot)}\otimes h_j)_{n\in\mathbb{Z}, j\in \mathbb{N}}$ is an orthonormal basis of $L^2(0,T;Y)$. Using Parseval's identity, we have
		\begin{align*}
			\left\|\Psi_T^\omega\right\|^2_2&= \sum_{k\in\mathbb{N}} \|\Psi_T^\omega e_k\|^2\\
			&= \sum_{k,j\in \mathbb{N}} \sum_{n\in\mathbb{Z}} \left\langle\Psi_T^\omega e_k, f_n\otimes h_j\right\rangle_{\mathcal{Y}_T}^2\\
			&=\sum_{k,j\in \mathbb{N}} \sum_{n\in\mathbb{Z}} \left(\int_0^T\langle(\Psi_T^\omega e_k)(t),\left(f_n\otimes h_j\right)(t)\rangle_Y dt\right)^2\\
			&=\sum_{k,j\in \mathbb{N}} \sum_{n\in\mathbb{Z}} \left(\int_0^T\langle e^{-\frac{2\pi ni}{T}t}(\Psi_T^\omega e_k)(t), h_j\rangle_Y dt\right)^2.
		\end{align*}
		Using the fact that $\left(\Psi_T^\omega e_k\right)(t)=e^{-\omega t}\left(\Psi_T e_k\right)(t)$, we have
		\begin{align*}
			\int_0^Te^{-\frac{2\pi n i}{T}t}\left(\Psi_T^\omega e_k\right)(t)dt&= \int_0^Te^{-\left(\omega+\frac{2\pi ni}{T}\right)t}\left(\Psi_T e_k\right)(t)dt\\
			&=CR\left(\omega+\frac{2\pi ni}{T},A\right)Je_k,
		\end{align*}
		where $J:=Id-e^{-\omega T}\mathbb{T}(T)$. Thus,
		\begin{align*}
			\left\|\Psi_T^\omega\right\|^2_2&= \sum_{k,j\in \mathbb{N}} \sum_{n\in\mathbb{Z}} \left(\left\langle CR\left(\omega+\frac{2\pi ni}{T},A\right)J e_k,h_j\right\rangle_Y \right)^2\\
			&= \sum_{n\in\mathbb{Z}} \left\| CR\left(\omega+\frac{2\pi ni}{T},A\right)J\right\|^2_2.
		\end{align*}
		Taking into account that the operator $J$ is invertible, since the spectral radius of $e^{-\omega T}\mathbb{T}(T)$ is less than one, it follows that \eqref{estim} holds if and only if $\Psi_T^\omega$ is Hilbert-Schmidt. This happens if and only if $\Psi_T$ is Hilbert-Schmidt. \\
		(ii) follows from Proposition \ref{prop3} and (i).
	\end{proof}
	
	(ii) of the above proposition is an extension of the result in \cite[Corollary 7.4]{Nerveen}  (where the case $B\in\mathcal{L}(U,H)$ was considered) in the Hilbert setting.
	
	\begin{ex}
		Let $U$ be a separable Hilbert space. We consider the following transport equation with boundary-noise
		\begin{align}\label{trans}
			\begin{cases}
				\frac{\partial X(t,\theta)}{\partial t}=\frac{\partial X(t,\theta)}{\partial \theta},&\theta\in[-r,0],\;\;t\in[0,T],\\
				X(0,\theta)=\varphi(\theta),&\theta\in[-r,0],\\
				X(t,0)=\dot{W}(t), &t\in[0,T],
			\end{cases}
		\end{align}
		where $\varphi\in H:=L^2([-r,0];U)$ and $(W(t))_{t\in[0,T]}$ is a cylindrical Wiener process over $U$.
		
		On the space $H$, we define the operators
		\begin{align*}
			&Q_m\varphi=\partial_\theta\varphi, \qquad \varphi\in D(Q_m)=W^{1,2}([-r,0];U),\\
			&G\varphi=\varphi(0), \qquad \varphi\in W^{1,2}([-r,0];U).
		\end{align*}
		
		Thus, the problem \eqref{trans} is reformulated in $H$ as the following
		\begin{align*}
			\begin{cases}
				\dot{X}(t)=Q_mX(t), & t\in[0,T],\\
				X(0)=\varphi,\\
				GX(t)=\dot W(t).
			\end{cases}
		\end{align*}
		It is clear that the operator $G:D(Q_m)\to U$ is surjective. Moreover, it is well known that the operator
		\begin{align*}
			Q:=Q_m, \qquad D(Q):=\left\{\varphi\in D(Q_m),\quad \varphi(0)=0\right\},
		\end{align*}
		generates the left shift semigroup $(S(t))_{t\geq0}$ on $H$ defined by
		\begin{align*}
			(S(t)\varphi)(\theta)=\begin{cases}
				0,& -t\leq\theta\leq0,\\
				\varphi(t+\theta), & -r\leq \theta<-t.
			\end{cases}
		\end{align*}
		By a simple calculations, the Dirichlet operator associated with $Q_m$ and $G$ is $\mathbb{D}_\lambda:U\to H$ given by
		\begin{align*}
			(\mathbb{D}_\lambda v)(\theta)=e^{\lambda\theta}v,\qquad \theta\in[-r,0],\;\;v\in U, \;\;\lambda\in\mathbb{C}.
		\end{align*}
		Thus, the control operator $B$ associated with $Q_m$ and $G$ is given by
		\begin{align*}
			B=\left(\lambda-Q_{-1}\right)\mathbb{D}_\lambda\in\mathcal{L}(U,H_{-1}), \qquad \lambda\in \mathbb{C}.
		\end{align*}
		Now, let $(u_k)_{k\in\mathbb{N}}$ be an orthonormal basis of $U$. Then, for any $\omega>\omega_0(S)$,
		\begin{align*}
			\sum_{n\in\mathbb{Z}}\left\|\mathbb{D}_{\omega+\frac{2\pi in}{T}}\right\|^2_2&=	\sum_{n\in\mathbb{Z}}	\sum_{k\in\mathbb{N}}\displaystyle\int_{-r}^0\left\|\left(\mathbb{D}_{\omega+\frac{2\pi in}{T}}u_k\right)(\theta)\right\|^2d\theta\\
			&=	\sum_{n\in\mathbb{Z}}	\sum_{k\in\mathbb{N}}\displaystyle\int_{-r}^0\left\|e^{\left(\omega+\frac{2\pi ni}{T}\right)\theta}u_k\right\|^2d\theta\\
			&=	\sum_{n\in\mathbb{Z}}	\sum_{k\in\mathbb{N}}\displaystyle\int_{-r}^0e^{2\omega\theta}d\theta=+\infty.
		\end{align*}
		It follows from Proposition \ref{car_adm}, that  $B$ is not $\mathcal{S}$-admissible for $Q$. Thus, the white-noise boundary problem \eqref{trans} does not have a solution in $H$.
	\end{ex}
	
	\section{Desch-Schappacher perturbation} \label{section4}
	In this section, we are interested in studying an unbounded perturbation result that concerns the existence of a solution for the perturbed problem.
	
	Let $\mathscr{P}\in\mathcal{L}(H,H_{-1})$ be an admissible control operator for $A$. Then, the operator
	\begin{align*}
		\mathscr{A}:=A_{-1}+\mathscr{P}, \qquad D(\mathscr{A})=\left\{x\in H,\quad (A_{-1}+\mathscr{P})x\in H\right\},
	\end{align*}
	generates a strongly continuous semigroup $\mathscr{T}:=\left(\mathscr{T}(t)\right)_{t\geq0}$ on $H$ given by
	\begin{align*}
		\mathscr{T}(t)x=\mathbb{T}(t)x+\int_0^t\mathbb{T}_{-1} (t-s)\mathscr{P}\mathscr{T}(s)xds
	\end{align*}
	for any $x\in H$. The operator $\mathscr{P}\in\mathcal{L}(H,H_{-1})$ is called a Desch-Schappacher perturbation for $A$,  see \cite{desch} and \cite[Chap. III, Corollary 3.4]{nagel}.
	
	For $\mathscr{P}\in\mathcal{L}(H,H_{-1})$, we shall consider  the perturbed stochastic Cauchy problem
	\begin{align*}
		\bf	(SCP)_{\mathscr{A},B}\;
		\begin{cases}
			dX(t)=\mathscr{A}X(t)dt+BdW(t), &t\in[0,T],\\
			X(0)=X_0.
		\end{cases}
	\end{align*}
	It follows from Theorem \ref{TH32}, that $\bf (SCP)_{\mathscr{A},B}$ has a solution if and only if the input-map
	\begin{align*}
		\Phi_T^\mathcal{\mathscr{P}}u=\int_0^T\mathscr{T}_{-1}(T-s)Bdu(s),\quad u\in\mathcal{U}_T,
	\end{align*}
	associated with $(\mathscr{A},B)$ is Hilbert-Schmidt from $\mathcal{U}_T$ to $H$. It is not immediately apparent how to establish the Hilbert-Schmidt property for $\Phi_T^\mathcal{\mathscr{P}}$ through a direct approach. This is due to the lack of a simple expression for the semigroup $\mathscr{T}_{-1}$. However, we can overcome this problem by using a dual argument. Specifically, we will rely on the characterization of the $\calS$-admissible control operators developed in the previous section.

	Let $P\in\mathcal{L}(H_1,H)$ be an admissible observation operator for $A$. It is well-known that the perturbed operator $\mathcal{A}:=A+P$ with domain $D(\mathcal{A})=D(A)$ generates a strongly continuous semigroup $\mathbb{S}:=(S(t))_{t\geq0}$ on $H$, see \cite{hadd2005unbounded}, see also \cite[Theorem 5.4.2]{tucsnak2009observation}. The following result shows the invariance of $\mathcal{S}$-admissibility for $A$ under the unbounded perturbation $P$.
	
	\begin{thm}\label{3.8}
		Let $C\in\mathcal{L}(H_1,Y)$ be an $\mathcal{S}$-admissible observation operator for $A$. Assume that $P\in\mathcal{L}(H_1,H)$ is an admissible observation operator for $A$. Then, $C$ is $\mathcal{S}$-admissible observation operator for $\mathcal{A}$.
	\end{thm}
	\begin{proof}
		Let $\omega>\max\left\{\omega_0(\mathbb{T}),\omega_0(\mathbb{S})\right\}$. Define the operators $A_\omega:=A-\omega$ and $\mathcal{A}_\omega:=A_\omega+P$ with domain $D(\mathcal{A}_\omega)=D(A_\omega)=D(A)$. Denote $\mathbb{T}_\omega$ and $\mathbb{S}_\omega$ the semigroup generated by $A_\omega$ and $\mathcal{A}_\omega$, respectively. By Remark \ref{rem1}, we have $\int_0^{+\infty}\|C_\Lambda \mathbb T_\omega(t)\|^2_2dt<\infty$ and $C_\Lambda \mathbb T_\omega(t)\in\mathcal{L}_2(H,Y)$ for a.e. $t\geq0$. Now, we define the operator
		\begin{align*}
			\mathfrak{S}:\mathbb{R}^+&\to \mathcal{L}_2(H,Y)\\
			t&\mapsto C_\Lambda\mathbb T_\omega(t).
		\end{align*}
		Thus, since $\mathcal{L}_2(H,Y)$ is a separable Hilbert space it follows from the Paley-Wiener theorem that the Laplace transform $\widehat{\mathfrak{S}}$ of $\mathfrak{S}$ is in $H^2\left(\mathcal{L}_2(H,Y)\right)$. Moreover, for any $x\in H$ and any $\lambda\in\mathbb{C}_0$, we have $\widehat{\mathfrak{S}(\lambda)}x=CR(\lambda,A_\omega)x$, and hence $ CR(\lambda,A_\omega)\in H^2\left(\mathcal{L}_2(H,Y)\right)$. On the other hand, we know that
		\begin{align} \label{yyyy}
			CR(\lambda,\mathcal A_\omega)=CR(\lambda,A_\omega)\left[I+PR(\lambda,\mathcal A_\omega)\right]
		\end{align}
for any $\lambda\in\mathbb{C}_0$.
		Moreover, $P$ is an admissible observation operator for $\mathcal{A}_\omega$, see \cite{hadd2005unbounded}. Thus,  it follows from the exponential stability of  $\mathbb{S}_\omega$ that $$\|PR(\lambda,\mathcal A_\omega)\|<\infty,$$  for any $\lambda\in\mathbb{C}_0$. According to \eqref{yyyy}, we obtain that $CR(\lambda,\mathcal A_\omega)\in  H^2\left(\mathcal{L}_2(H,Y)\right)$. Using again the Paley–Wiener theorem, we obtain that $\int_0^{+\infty}\|C_\Lambda\mathbb S_\omega(t)\|^2_2dt<\infty$. This means that $C$ is $\mathcal{S}$-admissible for $\mathcal{A}_\omega$, and hence also for $\mathcal{A}$.
	\end{proof}
	As a consequence of this result, we obtain the following perturbation result for $\bf(SCP)_{A,B}$ under Desh-Schappacher perturbation.
	\begin{cor}\label{cor}
		Let $\mathscr P\in\mathcal{L}(H,H_{-1})$ be an admissible control operator for $A$. If  $\bf(SCP)_{A,B}$ has a solution, then $\bf	(SCP)_{\mathscr{A},B}$ has a solution as well.
	\end{cor}
	\begin{proof}
According to Proposition \ref{prop3}, it suffices to prove that $B^\ast$ is $\calS$-admissible observation operator for $\mathscr{A}^\ast$. Observe that $\mathscr A^\ast=A^\ast+\mathscr{P}^\ast$ with domain $D(\mathscr{A}^\ast)=D(A^\ast)$. ( This is due to the fact that $\mathscr{A}$ is exactly the generator of the closed loop system associated with the regualr triple $(A,\mathscr{P},I)$ and the identity as admissible feedback. Thus $\mathscr{A}^\ast$ is the generator of the triple $(A^\ast,I,\mathscr{P}^\ast)$ with identity operator as admissible feedback.) Here we recall that $\mathscr{P}^\ast$ is an admissible observation operator for $A^\ast$. From Theorem \ref{TH32} and Proposition \ref{prop3} we deduce that $B^\ast$ is an $\mathcal S$-admissible observation operator for $A^\ast$. Therefore, Theorem \ref{3.8} implies that $B^\ast$ is an $\mathcal S$-admissible observation operator for $\mathscr{A}^\ast$. This ends the proof.
	\end{proof}
	
	As bounded operators $\mathscr{P}\in\mathcal{L}(H)$ are admissible control operators for $A$, we have the following result  as a consequence of Corollary \ref{cor}. It extends the result in \cite[Theorem 3.3]{haak} (where the case $B\in\mathcal{L}(U,H)$ was considered) in the Hilbert setting.
	\begin{cor}
		Let $\mathscr{P}\in\mathcal{L}(H)$. If $\bf	(SCP)_{A,B}$ has a solution, then $\bf	(SCP)_{A+\mathscr{P},B}$ has a solution as well.
	\end{cor}
\section{An application to a one-dimensional heat equation with a white-noise Neumann boundary control} \label{Last-section}
We consider the following heat equation with Neumann boundary control of white-noise type:
		
		\begin{align*}
			\bf{(HE)}\;\;
			\begin{cases}
				\frac{\partial X}{\partial t}(t,\xi)=\frac{\partial^2 X}{\partial \xi^2}(t,\xi), & 0<\xi<\pi,\; t\in[0,T],\\
				\frac{\partial X}{\partial \xi}(t,0)=MX(t,\cdot), &t\in[0,T],\\
				\frac{\partial X}{\partial \xi}(t,\pi)=\dot{W}(t), &t\in[0,T],\\
				X(0,\xi)=X_0(\xi), & 0<\xi<\pi,
			\end{cases}
		\end{align*}
		where $M\in\mathcal{L}\left(L^2(0,\pi),\mathbb{R}\right)$ and $(W(t))_{t\in[0,T]}$ is real Wiener process. To reformulate the system into the abstract setting, we  consider the following operators
		\begin{align*}
			&\mathscr{A}_m:=\frac{\partial^2}{\partial\xi^2}, \qquad D(\mathscr{A}_m):=\left\{\varphi\in H^2(0,\pi), \quad \varphi'(0)=M\varphi\right\},\\
			&\mathscr{G}:D(\mathscr{A}_m)\to \mathbb{R}, \qquad \mathscr{G}\varphi=\varphi'(\pi).
		\end{align*}
		With these operators, the problem $\bf(HE)$ is reformulated in $H:=L^2(0,\pi)$ as follows
		\begin{align}\label{exa}
			\begin{cases}
				\dot{X}(t)=\mathscr{A}_mX(t), &t\in[0,T],\\
				\mathscr{G}X(t)=\dot{W}(t), &t\in[0,T],\\
				X(0)=X_0.
			\end{cases}
		\end{align}
		In order to reformulate \eqref{exa} as a stochastic Cauchy problem, we verify assumptions $\bf(A1)$ and $\bf(A2)$. Since the operator $\mathscr{G}$ is surjective, it remains to verify that the operator
		\begin{align*}
			\mathscr{A}:=\mathscr{A}_m, \qquad D(\mathscr{A}):=\left\{\varphi\in H^2(0,\pi), \quad \varphi'(0)=M\varphi, \;\; \varphi'(\pi)=0\right\},
		\end{align*}
		generates a strongly continuous semigroup  on $H$. Let us first define the operators
		\begin{align*}
			&A_m:=\frac{\partial^2}{\partial\xi^2},\qquad	D(A_m):=\left\{\varphi\in H^2(0,\pi), \quad  \varphi'(\pi)=0\right\}, \\
			&G: D(A_m)\to \mathbb{R}, \qquad G\varphi=\varphi'(0).
		\end{align*}
		It is well-known that the operator
		\begin{align*}
			A:=A_m, \qquad D(A):=\left\{\varphi\in H^2(0,\pi), \quad \varphi'(0)=0, \;\; \varphi'(\pi)=0\right\},
		\end{align*}
		generates a strongly continuous semigroup $\mathbb{T}:=(\mathbb{T}(t))_{t\geq0}$ on $H$ and the operator $G$ is surjective.  Then, the Dirichlet operator $\mathbb{D}_\lambda:\mathbb{R}\to H$ associated with $A_m$ and $G$ satisfies
		\begin{align*}
			\mathbb{D}_\lambda\alpha=\varphi\; \Longleftrightarrow\; \left\{\lambda\varphi-\varphi''\;\;\text{in}\;[0,\pi], \quad\varphi'(0)=\alpha,\;\; \varphi'(\pi)=0 \right\}.
		\end{align*}
		Define $B:=-A_{-1}\mathbb{D}_0\in\mathcal{L}(\mathbb{R},H_{-1})$.
		\begin{lem}\label{lem}
			$B$ is an admissible control operator for $A$ and its adjoint $B^*\in\mathcal{L}(D(A),\mathbb{R})$ is given by
			\begin{align*}
				B^*\varphi=-\varphi(0).
			\end{align*}
		\end{lem}
		\begin{proof}
			See for example \cite[Chap. 3]{lasiecka}.
		\end{proof}
		We select  the operator $\mathscr{P}:=BM\in\mathcal{L}(H,H_{-1})$. We have the following lemma.
		\begin{lem}\label{lem1}
			\begin{enumerate}
				\item[(i)] 	$\mathscr{P}$ is an admissible control operator for $A$.
				\item[(ii)] The operator $\mathscr{A}$ coincides with the operator
				\begin{align*}
					\mathfrak{A}:=A_{-1}+\mathscr{P},\qquad D(\mathfrak{A})=\left\{\varphi\in H, \quad (A_{-1}+\mathscr{P})\varphi\in H\right\}.
				\end{align*}
				Moreover, it generates a strongly continuous semigroup $\mathscr{T}:=\left(\mathscr{T}(t)\right)_{t\geq0}$ on $H$.
			\end{enumerate}
		\end{lem}
		\begin{proof}
			(i) Let $\Phi_t^\mathcal{\mathscr{P}}$ be the input map associated with $A$ and $\mathscr{P}$
			\begin{align*}
				\Phi_t^\mathcal{\mathscr{P}}u:=\int_0^t\mathbb{T}_{-1}(t-s)\mathscr{P}u(s)ds, \qquad u\in\mathcal{U}_T.
			\end{align*}
			Since $B$ is admissible for $A$, we have
			\begin{align*}
				\left\|\Phi_t^\mathscr{P}u\right\|&=\left\|\int_0^t\mathbb{T}_{-1}(t-s)BMu(s)ds\right\|\\
				&\leq \gamma(t)\|M\|\left(\int_0^t\|u(s)\|^2ds\right)^{\frac{1}{2}}
			\end{align*}
			for any $u\in\mathcal{U}_T$. Thus, $\mathscr{P}$ is an admissible control operator for $A$.
			
			For (ii) see for example \cite{hadd2015}.
		\end{proof}
		Now, let  $\mathscr{D}_\lambda\in\mathcal{L}(\mathbb{R},H)$, $\lambda\in\rho(\mathscr{A})$, be the Dirichlet operator associated with $\mathscr{A}_m$ and $\mathscr{G}$. Then for any $\alpha\in\mathbb{R}$ and $\varphi\in H$ we have
		\begin{align*}
			\mathscr{D}_\lambda\alpha=\varphi\; \Longleftrightarrow\; \left\{\lambda\varphi-\varphi''\;\;\text{in}\;[0,\pi], \quad\varphi'(0)=M\varphi,\;\; \varphi'(\pi)=\alpha \right\}.
		\end{align*}
		We set $\mathscr{B}:=-\mathscr{A}_{-1}\mathscr{D}_0\in\mathcal{L}(\mathbb{R},H_{-1})$. Thus, the problem \eqref{exa} is reformulated as
		\begin{align}\label{exa2}
			\begin{cases}
				\dot{X}(t)=\mathscr{A}_{-1}X(t)+\mathscr{B}dW(t), &t\in[0,T],\\
				X(0)=X_0.
			\end{cases}
		\end{align}
		\begin{lem}\label{lem2}
			$\mathscr{B}$ is $\mathcal{S}$-admissible control operator for $A$.
		\end{lem}
		\begin{proof}
			First we remark that $D(\mathscr{A}^*)\supseteqq D(A^*)=D(A)$, since $A$ is self-adjoint. Using Lemmas \ref{lem} and \ref{lem1}, we get
			\begin{align*}
				\langle \mathscr{B}\alpha,\varphi\rangle&=-\langle \mathscr{A}_{-1}\mathscr{D}_0\alpha,\varphi\rangle\\
				&=-\langle \mathscr{D}_0\alpha,\mathscr{A}^*\varphi\rangle\\
				&=-\langle \mathscr{D}_0\alpha,A^*\varphi\rangle-\langle \mathscr{D}_0\alpha,M^*B^*\varphi\rangle\\
				&=-\langle \mathscr{D}_0\alpha,A\varphi\rangle+M\mathscr{D}_0\varphi(0)\alpha
			\end{align*}
			for any $\alpha\in\mathbb{R}$ and any $\varphi\in D(A)$. \\
			On the other hand, Using an integration by parts twice and the definition of $\mathscr{D}_0$, we get that
			\begin{align*}
				\langle \mathscr{D}_0\alpha,A\varphi\rangle= -\varphi(\pi)\alpha+M\mathscr{D}_0\varphi(0)\alpha.
			\end{align*}
			Thus,
			\begin{align*}
				\langle \mathscr{B}\alpha,\varphi\rangle=\varphi(\pi)\alpha,
			\end{align*}
			which means that the adjoint operator of $\mathscr B$ is given by
			\begin{align*}
				\mathscr{B}^*\phi=\phi(\pi), \quad\forall \phi\in D(A).
			\end{align*}
			
			We recall that the family $(\phi_n)_{n\in\mathbb{N}}$, defined by 	
			\begin{align*}
				\phi_n(x)=\sqrt{\frac{2}{\pi}}\cos\left(nx\right), \quad \forall n\in\mathbb{N}, \;\; x\in(0,\pi),
			\end{align*}
			is an orthonormal basis in $H$ formed by the eigenvectors of $A$ and the corresponding eigenvalues are $\lambda_n=-n^2$, with $n\in\mathbb{N}$. Thus, the semigroup $\mathbb{T}$ is given by
			\begin{align*}
				\mathbb{T}(t)f=\sum_{n\in\mathbb{N}}e^{-n^2t}\langle f,\phi_n\rangle\phi_n,
			\end{align*}
			and $\mathscr{B}^*$ is defined by $\mathscr{B}^*\phi_n=\sqrt{\frac{2}{\pi}}\cdot(-1)^n$. An easy calculation shows that $\mathscr{B}^*$ is an admissible observation operator for $A$. Then, By (ii) of Proposition \ref{prop3}, $B$ is $\mathcal{S}$-admissible for $A$ if and only
			\begin{align*}
				\int_0^T\|\mathscr{B}^*_\Lambda \mathbb{T}^*(t)\|^2_2dt<\infty.
			\end{align*}
			For any $T>0$, we have
			\begin{align*}
				\int_0^T\|\mathscr B^*_\Lambda \mathbb{T}^*(t)\|^2_2dt&=\sum_{n\in\mathbb{N}}\int_0^T|\mathscr B^*\mathbb{T}^*(t)\phi_n|^2dt\notag\\
				&=\sum_{n\in\mathbb{N}}\int_0^T|\mathscr B^*e^{-n^2t}\phi_n|^2dt\notag\\
				&=\frac{2}{\pi}\sum_{n\in\mathbb{N}}\int_0^Te^{-2n^2t}dt<\infty.
			\end{align*}
			Consequently, $\mathscr B$ is a $\mathcal{S}$-admissible control operator for $A$.
		\end{proof}
		\begin{thm}
			The problem $\bf(HE)$ has a unique mild solution $X$ given by
			\begin{align*}
				X(t)=\mathscr{T}(t)X_0+\int_0^t\mathscr{T}_{-1}(t-s)BdW(s), \qquad t\in[0,T],
			\end{align*}
			for any $X_0\in H$.
		\end{thm}
		\begin{proof}
			Using the Lemma \ref{lem1} and \eqref{exa2}, the problem $\bf(HE)$ is reformulated as
			\begin{align*}
				\begin{cases}
					dX(t)=\mathfrak{A}X(t)dt+\mathscr{B}dW(t), &t\in[0,T]\\
					X(0)=X_0,
				\end{cases}
			\end{align*}
			The result follows from Lemma \ref{lem1}, Lemma \ref{lem2} and Corollary \ref{cor}.
		\end{proof}





\end{document}